\documentclass[12pt]{article}
\usepackage{amssymb,amsmath,amscd,epsfig,amsthm,multirow}
\usepackage{hyperref}

\let\ssection=\section
\renewcommand{\section}{\setcounter{equation}{0}\ssection}

\newcommand{\Fix}{\mathop{\rm Fix}\nolimits}
\newcommand{\Stf}{\mathop{\rm Stab}\nolimits_F}
\newcommand{\Stg}{\mathop{\rm Stab}\nolimits_G}

\theoremstyle{plain}
\newtheorem{theorem}{Theorem}[section]
\newtheorem{corollary}[theorem]{Corollary}
\newtheorem{proposition}[theorem]{Proposition}

\theoremstyle{definition}
\newtheorem{definition}[theorem]{Definition}
\newtheorem{problem}[theorem]{Problem}

\newtheorem{notation}[theorem]{Notation}
\newtheorem{remark}[theorem]{Remark}

\newtheorem*{theoremirrational}{Theorem~\ref{thm_schreier_irrational}}
\newtheorem*{theoremconjug}{Theorem~\ref{cor_action_conjug}}
\newtheorem*{theoremrational}{Theorem~\ref{thm_schreier_rational}}

\begin{document}
\title{Schreier graphs of actions of Thompson's group $F$ on the unit interval and on the Cantor set}
\author{Dmytro Savchuk\\ University of South Florida}

\maketitle

\abstract{Schreier graphs of the actions of Thompson's group $F$ on the orbits of all points of the unit interval and of the Cantor set with respect to the standard generating set $\{x_0,x_1\}$ are explicitly constructed. The closure of the space of pointed Schreier graphs of the action of $F$ on the orbits of dyadic rational numbers and corresponding Schreier dynamical system are described. In particular, we answer the question of Grigorchuk on the Cantor-Bendixson rank of the underlying space of the Schreier dynamical system in the context of $F$. As applications we prove that the pointed Schreier graphs of points from $(0,1)$ are amenable, have infinitely many ends, and are pairwise non-isomorphic. Moreover, we prove that points $x,y\in(0,1)$ have isomorphic non-pointed Schreier graphs if and only if they belong to the same orbit of $F$.}

\section*{Introduction}

Thompson's group $F$ was first defined by Richard Thompson in 1965 in the context of his work on logic. It was later used in~\cite{mckenzie_t:unsolvable} in the construction of finitely-presented groups with unsolvable word problem. The related Thompson's groups $T$ and $V$ were the first examples of finitely presented simple groups. Since then $F$ was rediscovered several times as a group arising naturally in different contexts. It has connections to homotopy theory, group-based cryptography, the four-color theorem, and many other areas~\cite{cannon_fp:intro_thompson}.

In this paper we study Schreier graphs related to $F$. These graphs are generalizations of the Cayley graph of $G$, which are constructed for each choice of a generating set for $G$ and a subgroup of $G$. Another equivalent definition of Schreier graphs comes from the action of $G$ on some set. Corresponding definitions are given in Section~\ref{sec_schreier_dyadic}.

Studying Schreier graphs of $F$ has several motivations. Schreier graphs have been used occasionally as of now for more than 80 years (sometimes under the name Schreier coset diagrams~\cite{coxeter_m:generators_and_relators}). But recently they became important in the relation to problems coming from analysis~\cite{bartholdi_g:spectrum}, holomorphic dynamics~\cite{nekrash:self-similar}, ergodic theory~\cite{grigorch_kn:ergodic}, probability~\cite{kapovich:schreier}, etc. Most of the recent constructions and applications of Schreier graphs are related to groups acting on rooted trees by automorphisms (see, for example,~\cite{gns00:automata,grigorchuk-s:hanoi-cr,bondarenko:phd}). Every such action induces a sequence of Schreier graphs arising from the action on the levels of the tree, and an uncountable family of Schreier graphs of the action of the group on the orbits of elements of the boundary of the tree.

In~\cite{dangeli_dmn:schreier} the Schreier graphs of the action of the so-called Basilica group on the orbits of elements of the boundary of the tree are completely classified. The Basilica group which can be defined as an iterated monodromy group of a self-map of a complex plane $z\mapsto z^2-1$. In~\cite{nekrash:self-similar} it is shown that the Schreier graphs of iterated monodromy groups arising from the actions on rooted trees converge to the Julia sets of corresponding self-coverings. Another recent paper~\cite{bond_cdn:amenable} completely classified up to isomorphism Schreier graphs of the action of a group generated by two nontrivial states of a 3-state automaton over binary alphabet on the orbits of elements of the boundary of the tree. Surprisingly, even though the original group has exponential growth, all obtained Schreier graphs have intermediate growth. Two of these graphs were shown to be automatic in~\cite{miasnikov_s:automatic_graph}, providing the first examples of automatic graphs of intermediate growth.

Some time ago Zimmer~\cite{zimmer:ergodic_theory},
and more recently Vershik~\cite{vershik:nonfree} and Grigorchuk~\cite{grigorch:dynamics11}, in the context of ergodic theory, switched the attention to Schreier dynamical systems, where the group acts on a compact subspace in the space of pointed Schreier graphs. There are two equivalent definitions of these dynamical systems. One can consider the action of the group on the set of pointed Schreier graphs by shifting the selected vertex over the edges. Equivalently, the group acts on the set of its subgroups by conjugation, and one can study the Schreier graphs of this action. The precise definition is given in Section~\ref{sec_closure}. Interestingly enough, sometimes it is possible to show that given an action of a group on a set, the Schreier dynamical system constructed from just one orbit, can recover the original action of the group on the whole set. For example, Vorobets in~\cite{vorobets:schreier_of_grigorchuk12} proved this phenomenon for Schreier dynamical systems arising from the action of Grigorchuk group on the boundary of binary tree.
We show in Theorem~\ref{cor_action_conjug} that this is also the case for the action of $F$ on $[0,1]$:

\begin{theoremconjug}
The action of $F$ on the perfect kernel $\mathcal D$ of the Schreier dynamical system is topologically conjugate to the standard action of $F$ on the Cantor set $\{0,1\}^\omega$.
\end{theoremconjug}

Schreier graphs also allow us better understand the actions of $F$. For example, Schreier graphs allow us to obtain bounds on the length of elements of $F$, give a simple way to construct elements of $F$ with prescribed action on $[0,1]$, better understand the dynamics of elements of $F$, etc.

Another motivation to study Schreier graphs is their direct relation to amenability of the group. A group is nonamenable if and only if some (all) of its Cayley graphs are nonamenable. But it is often hard to construct Cayley graphs. On the contrary, Schreier graphs can be  described explicitly in many situations and the nonamenability of any Schreier graph of a group implies nonamenability of a group. Thus, to prove nonamenability of $F$ it suffices to construct a nonamenable Schreier graph of $F$. This would provide an answer to one of the most intriguing and compelling open questions about $F$: is $F$ amenable? This question was first posed by Richard Thompson in 1960s, but gained attention after Geoghegan popularized it in 1979 (see p.549 of~\cite{gersten_s:comb_group_theory}). It is with no doubt one of the biggest driving forces for research related to $F$. The initial motivation for this question was an attempt to construct a finitely generated non-amenable group without non-abelian free subgroups, thus providing a counterexample to the von~Neumann-Day problem~\cite{day:means50}. The fact that $F$ does not have non-abelian free subgroups was proved by Brin and Squier in~\cite{brin_s:piecewise}. Since the question was posed, examples of such groups have been already found: Adian proved in~\cite{adian:nanf} that the free Burnside groups are nonamenable, later Ol'shanskii and Sapir in~\cite{olshanskii_s:non-amenable} constructed a quite non-trivial finitely presented example. Very recently, Lodha and Moore~\cite{lodha_m:geometric_von_neumann}, based on a more general construction of Monod~\cite{monod:groups_of_PL_homeos13}, constructed another finitely presented example which additionally is torsion-free. This example lives in the group of all piecewise projective homeomorphisms of $\mathbb R$ and thus has a lot in common with Thompson's group $F$: according to the unpublished result of Thurston (see~\cite{lodha_m:geometric_von_neumann} for details) $F$ also can be realized as a subgroup of this group. But the question of amenability of $F$ is now mainly interesting on its own. Definitely, Thompson's group $F$ is the most famous group whose amenability is not determined yet.

Thompson's group $F$ does not act by automorphisms on the rooted tree, but it acts by homeomorphisms on the boundary of the binary tree, which is homeomorphic to the Cantor set. This action agrees with the action on the unit interval and it induces the family of Schreier graphs of the action of $F$ on the orbits of elements of the Cantor set.

This paper is a natural extension of paper~\cite{savchuk:thompson}, in which the Schreier graph of the action of $F$ on the dyadic rational numbers is constructed. The question that remained unanswered in~\cite{savchuk:thompson} is about the structure of Schreier graphs corresponding to the action of $F$ on the orbits of points in $[0,1]$ which are not dyadic rationals. In this paper we completely answer this question by explicitly constructing these Schreier graphs based on the binary expansion of a point $x\in[0,1]$.

For simplicity, in this paper we will call the Schreier graph of the action of $F$ on the orbit of a point $x\in[0,1]$ by simply the Schreier graph of $x$ and denote it by $\Gamma_x$. Sometimes there will be need to specify the base vertex in a Schreier graph and, with a slight abuse of notation, the phrase ``the Schreier graph of $x$'' will also mean ``the pointed Schreier graph of the action of $F$ on the orbit of a point $x$ with selected vertex $x$''.

Since dyadic rational numbers are dense in $[0,1]$, for each point $x$ one can choose a sequence of dyadic rationals converging to $x$ and one might expect that the sequence of corresponding Schreier graphs will converge to the Schreier graph of $x$ in a certain suitable sense. It turns out that this is exactly the case when $x$ is irrational, which enables us to construct the Schreier graphs of the actions of $F$ on the orbits of irrational points in $[0,1]$:

\begin{theoremirrational}
Each irrational number $x$ in $[0,1]$ can be uniquely written as either $0.1^n0w$ or $0.0^m1w$ for some infinite word $w$ over $\{0,1\}$.
The Schreier graph $\Gamma_x$ of the action of $F$ on the orbit of $x$ is depicted in Figure~\ref{fig_correspondence} (particularly for $x=0.0001w$, where $w=1001\ldots$).
\end{theoremirrational}

However, we cannot use the same method for rational numbers. This phenomenon occurs because for every rational number $q\in(0,1)$ (not necessarily dyadic rational) there is an element of $F$ that fixes $q$ but does not fix any neighborhood of $q$. This simple fact is of a folklore type, but is somewhat counterintuitive. For the sake of completeness we give its proof here (see Proposition~\ref{prop_rationals}).

In order to understand Schreier graphs of rational numbers in $[0,1]$ we still use the Schreier graph of the action on dyadic rationals. By changing the selected point in this Schreier graph we obtain a family $\mathcal C$ of pointed graphs. Thompson's group $F$ acts on this set by shifting the basepoint according to the labels on the edges. By continuity this action can be extended to the closure $\overline{\mathcal C}$ of the set $\mathcal C$ in the pointed graphs topology (sometimes called Gromov-Hausdorff topology), in which two pointed marked graphs are close if they have isomorphic balls of large radius centered at the selected vertices (see \cite{gromov:structures}, Chapter 3 for the general construction, and~\cite{grigorch:degrees} for the first appearance in the context of pointed graphs). It turns out that after removing isolated points from this closure, one obtains a set of pointed marked graphs which is homeomorphic to the Cantor set. This answers the question of Grigorchuk (Question~6.1 in~\cite{grigorch:dynamics11eng}) in the context of $F$.  Moreover, this set is invariant under the action of $F$, and the restriction of the action of $F$ to this set is conjugate to the standard action of $F$ on the Cantor set. Now using this interpretation of the standard action it is not too hard to understand the Schreier graphs of rational points.

\begin{theoremrational}
Each rational point $x$ of the Cantor set $\{0,1\}^\omega$ except $000\ldots$ and $111\ldots$ can be uniquely written as either $x=1^n0vwww\ldots$ or $x=0^m1vwww\ldots$ for some finite words $v,w$ over $\{0,1\}$ such that $w$ is not a proper power and the ending of $v$ differs from the one of $w$. The Schreier graph $\Gamma_x$ of the action of $F$ on the orbit of $x$ is depicted in Figure~\ref{fig_correspondence_rational}.
\end{theoremrational}

The structure of the paper is as follows. In Section~\ref{sec_thompson_group} the definition and basic facts about Thompson's group are given. Section~\ref{sec_schreier_dyadic} contains a description from~\cite{savchuk:thompson} of the Schreier graph of the action of $F$ on the set of dyadic rational numbers from the interval $(0,1)$. The Schreier graphs of the action of $F$ on the orbits of irrational numbers are constructed in Section~\ref{sec_schreier_irrational}. The closure of the space of pointed Schreier graphs of dyadic rational numbers is studied in Section~\ref{sec_closure}. This finally allows us to give a complete description of Schreier graphs of all rational numbers and to provide some applications in Section~\ref{sec_schreier_rational}.

The author would like to express deep gratitude to Rostislav Grigorchuk for bringing my attention to the world of Thompson's groups and for continued attention to the project; and to Matthew Brin and Lucas Sabalka for careful reading of the draft and valuable comments that enhanced the paper. I also indebted to the anonymous referee for thorough reading of the manuscript and providing the precious generous feedback.

\section{Thompson's group}
\label{sec_thompson_group}

\begin{definition}
\emph{Thompson's group} $F$ is the group of all strictly
increasing piecewise linear homeomorphisms from the closed unit
interval $[0,1]$ to itself that are differentiable everywhere except
at finitely many dyadic rational numbers and such that on the
intervals of differentiability the derivatives are integer powers of
$2$. The group operation is composition of homeomorphisms.
\end{definition}

Basic facts about this group can be found in the survey
paper~\cite{cannon_fp:intro_thompson}. In particular, it is proved
that $F$ is generated by two homeomorphisms $x_0$ and $x_1$ given by
$$x_0(t)=\left\{
\begin{array}{ll}
\frac t2,&0\leq t\leq\frac12,\\
t-\frac14,&\frac12\leq t\leq\frac34,\\
2t-1,&\frac34\leq t\leq1,\\
\end{array} \right.
\qquad x_1(t)=\left\{
\begin{array}{ll}
t,&0\leq t\leq\frac12,\\
\frac t2+\frac14,&\frac12\leq t\leq\frac34,\\
t-\frac18,&\frac34\leq t\leq\frac78,\\
2t-1,&\frac78\leq t\leq1.\\
\end{array} \right.
$$

The graphs of $x_0$ and $x_1$ are displayed in
Figure~\ref{fig_gens}.
\begin{figure}[h]
\begin{center}
\includegraphics{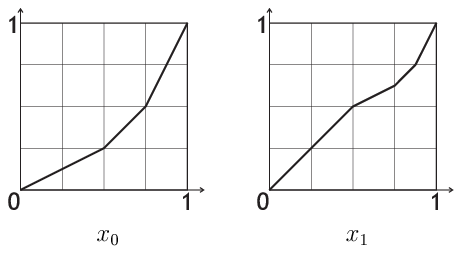}
\end{center}
\caption{Generators of $F$\label{fig_gens}}
\end{figure}

Throughout the paper we will use the following convention.

\begin{notation}
For any two elements $f$, $g$ of $F$ and any $x\in[0,1]$
\begin{equation}
\label{eqn_conv}
(fg)(x)=g(f(x)),\quad f^g=gfg^{-1}.
\end{equation}
\end{notation}

Sometimes it will be more convenient to consider the induced action of $F$ on the Cantor set. Each point of the interval $[0,1]$ can be associated with its binary expansion. This allows us to extend naturally the action of Thompson's group $F$ on the sets $X^\ast$ and $X^\omega$ of all finite and infinite words over the alphabet $X=\{0,1\}$ respectively. The latter set is homeomorphic to a Cantor set and it is easy to see that $F$ acts on it by homeomorphisms.  The generators $x_0$ and $x_1$ can be defined as follows:

\begin{equation}
\label{eqn_cantor_action}
\begin{array}{ccc}
x_0:\ \left\{\begin{array}{lll}
0w&\mapsto& 00w,\\
10w&\mapsto& 01w,\\
11w&\mapsto& 1w,
\end{array}\right.&\hspace{1.5cm}&x_1:\ \left\{\begin{array}{lll}
0w&\mapsto& 0w,\\
10w&\mapsto& 100w,\\
110w&\mapsto& 101w,\\
111w&\mapsto& 11w,
\end{array}\right.\\
\end{array}
\end{equation}
where $w$ is an arbitrary word in $X^\omega$. These homeomorphisms can be defined by finite state asynchronous automata. We will refer to the above action of $F$ as to the \emph{standard} action of $F$ on $X^\ast$.

Note that the dyadic rational numbers in $[0,1]$ correspond to the sequences ending in $000\ldots$ or $111\ldots$, and rational points in $[0,1]$ correspond to the preperiodic sequences $vw^{\infty}$ in $X^\omega$. We will call such sequences \emph{rational} elements of $X^\omega$. All other elements we will call \emph{irrational}. There is a one-to-one correspondence between irrational elements of $X^\omega$ and irrational numbers in $[0,1]$.

The main purpose of this paper is to describe the action of $F$ on $[0,1]$ via the Schreier graphs of this action restricted to the orbits of points in $[0,1]$.

\begin{definition}
Let $G$ be a group generated by a finite generating set $S$ acting on a set $M$. The \emph{(orbital) Schreier graph} $\Gamma(G,S,M)$ of the action of $G$ on $M$ with respect to the generating set $S$ is an oriented labeled graph defined as follows. The set of vertices of $\Gamma(G,S,M)$ is $M$ and there is an arrow from $x\in M$ to $y\in M$ labeled by $s\in S$ if and only if $x^s=y$.
\end{definition}

An equivalent alternative view of Schreier graphs goes back to Schreier, who called these graphs \emph{coset graphs}. For any subgroup $H$ of $G$, the group $G$ acts on the right cosets in $G/H$ by right multiplication. The corresponding Schreier graph $\Gamma(G,S,G/H)$ is denoted as $\Gamma(G,S,H)$ or just $\Gamma(G,H)$ if the generating set is clear from the context.

Conversely, if $G$ acts on $M$ transitively, then $\Gamma(G,S,M)$ is canonically isomorphic to $\Gamma(G,S,\Stg(x))$ for any $x\in M$, where the vertex $y\in M$ in $\Gamma(G,S,M)$ corresponds to the coset from $G/\Stg(x)$ consisting of all elements of  $G$ that move $x$ to $y$.  Also, to simplify notation, we will call $\Gamma(G,S,\Stg(x))$ simply the Schreier graph of $x$ and denote by $\Gamma_x$ when the group, the set and the action are clear from context.

In the subsequent sections we will completely describe the Schreier graphs of the action of $F$ on the orbits of all points of $[0,1]$. It is also important to mention that each dyadic rational point $0.w10^\infty$ in $[0,1]$ corresponds to two points in the Cantor set $X^\omega$: $w10^\infty$ and $w01^\infty$. The Schreier graphs of the action of $F$ on the orbits of these two points are isomorphic to the Schreier graph of the action of $F$ on the orbit of corresponding point in $[0,1]$ as pointed marked graphs. However, they have different labeling of vertices.

From the second viewpoint on the Schreier graphs, we construct the Schreier graphs with respect to stabilizers of points in $(0,1)$.
Since the ``complexity'' of the Schreier graph of $F$ with respect to a subgroup $H$ decreases as $H$ increases, it is natural to start from the description of Schreier graphs with respect to maximal subgroups of $F$. The following proposition shows that the stabilizers of points in $(0,1)$ are exactly of this kind.

\begin{proposition}
For each $x\in (0,1)$ the subgroup $\Stf(x)$ is maximal in $F$.
\end{proposition}

\begin{proof}

This proposition is a generalization of Proposition~2 in~\cite{savchuk:thompson} that considered only the case of dyadic rational numbers. The proof uses the same idea, but requires one extra step.

Let $x\in (0,1)$ be arbitrary, and let $f$ be any element from $F\setminus\Stf(x)$. Then for any $g\in F$ we show that $g\in\langle\Stf(x), f\rangle$. Let $g$ be an arbitrary element in $F$ that does not stabilize $x$.

Denote $v=g(x)$. Then exactly one of $f(x)$ and $f^{-1}(x)$ lies on the same side of $x$ as $v$. We denote this number by $u$. Without loss of generality assume $f(x)=u$, and both $u$ and $v$ are less than $x$. We will construct an element $h\in\Stf(x)$ such that $h(u)=v$. First, we note that $(f^{-1}g)(u)=g(f^{-1}(u))=g(x)=v$. So there is some element $h_1=f^{-1}g\in F$ satisfying $h_1(u)=v$.

Pick two arbitrary dyadic rational numbers $a$ and $b$ in the following way. Since $h_1(u)=v<x$, we can choose $a$ such that
\[u\leq a<h_1^{-1}(x)\]
and then we choose $b$ satisfying
\[\max\{a,h_1(a)\}<b\leq x.\]
Since $F$ acts transitively on the pairs of dyadic rational numbers and we have simultaneously $a<b$ and $h_1(a)<b$, there is $h_2\in F$ such that $h_2(a)=h_1(a)$ and $h_2(b)=b$.

Finally, the element
\[h(t)=\left\{
\begin{array}{l}
h_1(t),\quad 0\leq t\leq a,\\
h_2(t),\quad a< t\leq b,\\
t,\quad b< t\leq 1
\end{array}\right.\]
stabilizes $x$ since $x>b$ and satisfies $h(u)=v$ because $u<a$ and $h_1(u)=v$.

Now, the element $\tilde f=fh$ belongs to $\langle\Stf(x), f\rangle$ and satisfies $\tilde f(x)=v$. Thus, for $\tilde h=g\tilde f^{-1}$ we have $\tilde h(x)=\tilde f^{-1}(g(x))=\tilde f^{-1}(v)=x$. Thus $\tilde h\in\Stf(x)$ and $g=\tilde h\tilde f\in\langle\Stf(x), f\rangle$.
\end{proof}

The above proposition naturally raises the following question.

\begin{problem}
Are there any other maximal subgroups of $F$ except the stabilizers of singletons?
\end{problem}

\section{Schreier graphs of dyadic rational numbers}
\label{sec_schreier_dyadic}

The Schreier graph of the action of $F$ on the orbit of $1/2\in[0,1]$ was completely described in~\cite{savchuk:thompson}. Note that this orbit consists of all dyadic rational numbers. Under the standard action of $F$ on the Cantor set this graph describes the action on the orbit of $1000\ldots\in X^\omega$. This orbit consists of all sequences over $X$ that contain only finitely many (but more than zero) 1's.

\begin{proposition}[\cite{savchuk:thompson}]
\label{prop_schreier_dyadic}
The Schreier graph $\Gamma$ of the action of $F$ on the orbit of $\frac12$ consisting of all dyadic rational numbers is depicted in Figure~\ref{fig_schreier_dyadic}, where the following notation is used:
\begin{itemize}
\item a vertex labelled by a finite word $v$ over $\{0,1\}$ corresponds to the point $x=0.v0^\infty\in(0,1)$ (e.g., vertex labelled by $1$ corresponds to $0.10^\infty=\frac12$);
\item black vertices correspond to dyadic rational numbers in the interval $\bigl(0,\frac12\bigr)$;
\item grey vertices correspond to dyadic rational numbers in the interval $\bigl[\frac12,\frac34\bigr)$;
\item white vertices in the tree correspond to dyadic rational numbers in the interval $\bigl[\frac34,1\bigr)$;
\item dashed arrows correspond to the action of generator $x_0$;
\item solid arrows correspond to the action of generator $x_1$.
\end{itemize}

\begin{figure}[ht]
\hspace{1cm}\epsfig{file=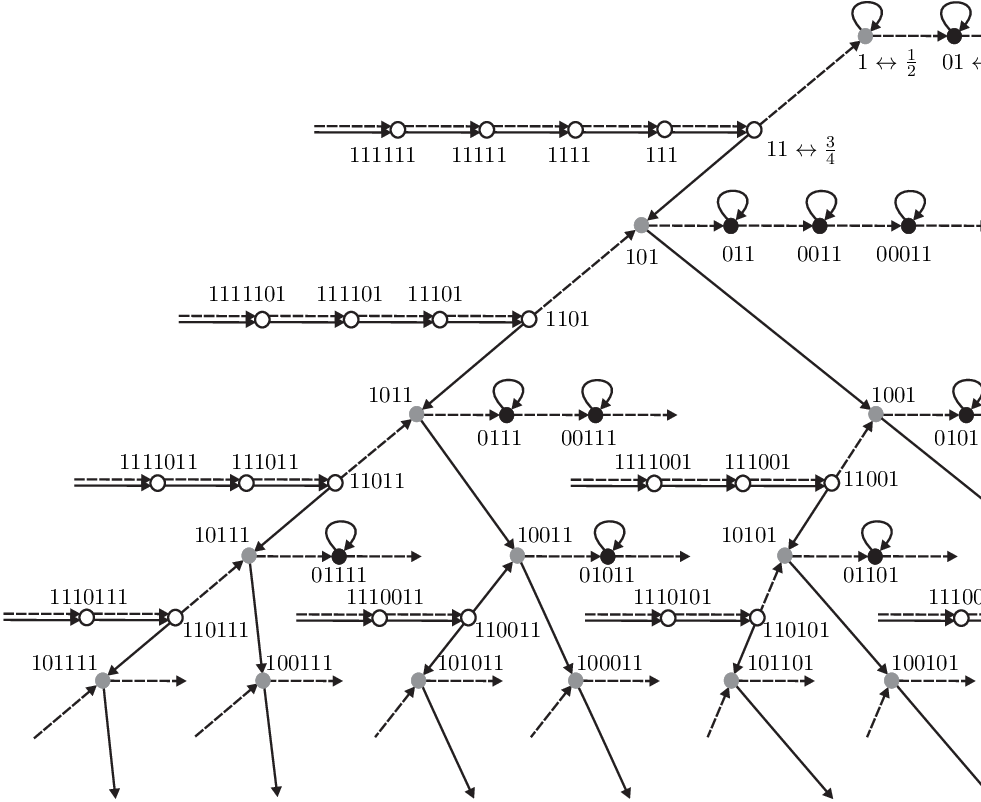,width=330pt}
\caption{Schreier graph $\Gamma$ of the action of $F$ on the orbit of $1000\ldots$\label{fig_schreier_dyadic}}
\end{figure}
\end{proposition}

Note that in~\cite{savchuk:thompson} the vertices of the graph are labeled by dyadic fractions, and not by their binary expansions. However, it is more natural to label the vertices as in Figure~\ref{fig_schreier_dyadic} because this labeling agrees with the standard action of $F$ on the Cantor set specified in~\eqref{eqn_cantor_action}. Namely, the rules describing how to get the label of the vertex from the label of the neighboring vertex follow from~\eqref{eqn_cantor_action} and for convenience are described in Table~\ref{tab_labeling}.

\begin{table}
\begin{center}
\begin{tabular}{|c|c|c|c|c|}
\hline
\multicolumn{4}{|c|}{\textbf{A move in the graph}}&\multirow{2}{*}{\textbf{Change of vertex label}}\\ \cline{1-4}
initial vertex&terminal vertex&direction&generator&\\\hline
\emph{black}&\emph{black}&$\rightarrow$&$x_0$&$0w\longrightarrow00w$\\\hline
\emph{grey}&\emph{black}&$\rightarrow$&$x_0$&$10w\longrightarrow01w$\\\hline
\emph{white}&\emph{grey}&$\nearrow$&$x_0$&$110w\longrightarrow10w$\\\hline
\emph{white}&\emph{white}&$\rightarrow$&$x_0,x_1$&$111w\longrightarrow 11w$\\\hline
\emph{white}&\emph{grey}&$\swarrow$&$x_1$&$110w\longrightarrow101w$\\\hline
\emph{grey}&\emph{grey}&$\searrow$&$x_1$&$10w\longrightarrow100w$\\\hline
\emph{black}&\emph{black}&none&$x_1$&$0w\to 0w$\\\hline
\end{tabular}
\caption{Changing the vertex labels while moving along the edges in the Schreier graphs\label{tab_labeling}}
\end{center}
\end{table}

Below, we will need to address the pointed Schreier graphs. First of all, we set up the notation. The pointed Schreier graph of a dyadic rational point $q$ with selected vertex labeled by $q$ will be denoted by $\Gamma_q$ and called simply Schreier graph of $q$.

\section{Schreier graphs of irrational numbers}
\label{sec_schreier_irrational}

For any $f\in F$ we denote by $\Fix(f)$ the set of points from
$[0,1]$ stabilized by $f$. Also, for any set $A\subset\mathbb R$ we denote by $\partial A$  the boundary of $A$ with respect to standard topology of $\mathbb R$. It is a folklore that elements of $F$ can nontrivially intersect the line $y=x$ at points whose coordinates are not dyadic rationals. For example, in~\cite{kassabov_m:simultaneous_conjugacy12} Kassabov and Matucci give an example of such an element and provide more information on where these coordinates could be in a bit more general settings. We prove the following proposition mainly for the purpose of integrity.

\begin{proposition}~\\[-0.5cm]
\label{prop_rationals}
\begin{itemize}
\item[($a$)] For each $f\in F$, $\partial\Fix(f)\subset \mathbb Q$.
\item[($b$)] Each rational number belongs to $\partial\Fix(f)$ for some $f\in F$.
\end{itemize}
\end{proposition}

\begin{proof}
\noindent($a$) Each point in $\partial\Fix(f)$ is either the breakpoint of $F$, in which case it is a dyadic rational number, or it is the $x$-coordinate of the intersection of a graph of $f$ and the line $y=x$. Since the equation of all the line segments of the graph of $f$ involve only rational coefficients, both coordinates of the intersection point must be rational.

\noindent($b$) For dyadic rational numbers the claim is obvious, since these points can be the breakpoints of elements of $F$.

Let $\frac{k}{2^t\cdot l}\in [0,1]$ be arbitrary rational number so that $\gcd(k,2^t\cdot l)=1$ and $l$ is odd and greater than one. Consider a line $L$ given by the equation
\begin{equation*}
y=2^{\phi(l)}\bigl(x-\frac{k}{2^t\cdot l}\bigr)+\frac{k}{2^t\cdot l},
\end{equation*}
where $\phi$ denotes the Euler function. This line passes through $\bigl(\frac{k}{2^t\cdot l}, \frac{k}{2^t\cdot l}\bigr)$ and has a slope that is a power of $2$. On the other hand this line intersects the $x$-axis at the point $\left(\frac{k}{2^{\phi(l)+t}}\cdot\frac{2^{\phi(l)}-1}{l},0\right)$, both of whose coordinates are dyadic rational numbers since $l|(2^{\phi(l)}-1)$. Therefore, for any point of line $L$ with dyadic rational $x$-coordinate, its $y$-coordinate also will be dyadic rational. Hence, we can choose two points $P_0(x_0,y_0)$ and $P_1(x_1,y_1)$ on the line $L$, whose coordinates are dyadic rational numbers from the interval $(0,1)$ and such that the point $\bigl(\frac{k}{2^t\cdot l}, \frac{k}{2^t\cdot l}\bigr)$ lies on the segment $P_0P_1$.

\begin{figure}
\begin{center}
\epsfig{file=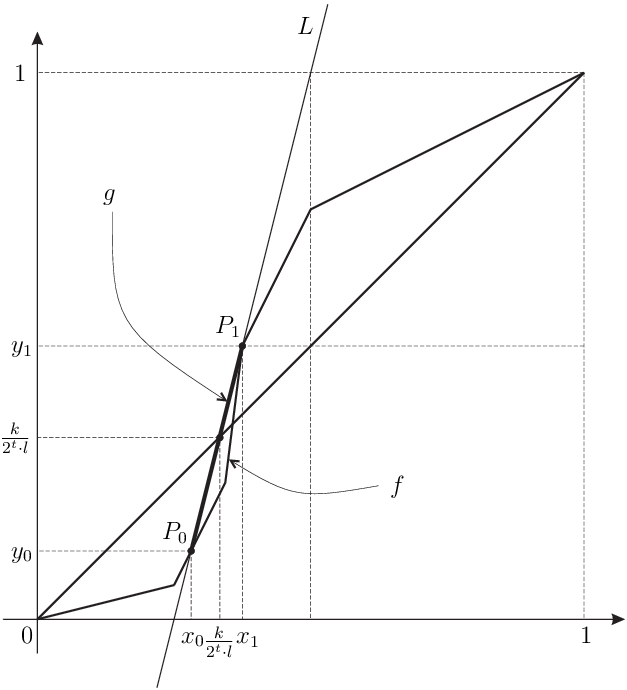, width=250pt}
\parbox{5in}{\caption{Constructing an element of $F$ which intersects the
diagonal at a given rational point\label{fig_surgery}}}
\end{center}
\end{figure}

It is well known (see~\cite{cannon_fp:intro_thompson}) that $F$ acts transitively on the pairs of dyadic rational numbers. Therefore, there is an element $f\in F$ satisfying $f(x_0)=y_0$ and $f(x_1)=y_1$. We make a surgery by cutting the piece of the graph of $f$ between $x_0$ and $x_1$ and replacing this piece by the corresponding segment of the line $L$. By doing this we construct a piecewise linear homeomorphism $g$ from $[0,1]$ to itself, which is order preserving, whose breakpoints are dyadic rational numbers and whose slopes are powers of $2$ (see Figure~\ref{fig_surgery}). In other words we get that $g$ is an element of $F$ satisfying $g\bigl(\frac k{2^t\cdot l}\bigr)=\frac k{2^t\cdot l}$ and $g(x)\neq x$ in some neighborhood of $\frac{k}{2^t\cdot l}$. Thus $\frac k{2^t\cdot l}\in \partial\Fix(g)$.
\end{proof}

Let $T$ denote the space of all regular marked pointed graphs (with marked edges and selected base vertex). This space is naturally endowed with pointed graphs topology (\cite{gromov:structures}, Chapter 3), and was used for the first time in the context of pointed graphs by Grigorchuk in~\cite{grigorch:degrees}. This topology is induced by the following metric. For $(\Gamma,x), (\widetilde\Gamma,y)\in\mathcal T$
we define
\[d(\bigl(\Gamma,x), (\widetilde\Gamma,y)\bigr)=2^{-n},\]
where $n$ is the largest integer for which balls $B_{\Gamma}(x,n)$ and $B_{\widetilde\Gamma}(y,n)$ are isomorphic as marked pointed graphs (i.e. two graphs are close in this topology if they have isomorphic balls of large radius centered at the selected points).

The proposition that follows, together with the description of Schreier graphs of dyadic rational numbers in Proposition~\ref{prop_schreier_dyadic}, allows us to describe completely the Schreier graphs of irrational numbers from $[0,1]$.

\begin{proposition}
\label{prop:psi_continuous}
Let $\psi:\ [0,1]\to\mathcal T$ be the map defined by
\[\psi(x)=\Gamma(F,\Stf(x),\{x_0,x_1\}).\]
The map $\psi$ is continuous at all irrational points of
$[0,1]$.\end{proposition}

\begin{proof}
For any element $f\in F$ the set $\partial\Fix(f)$ is finite.
Therefore the sets
\[D_n=\bigcup_{f\in B_F(id,n)}\partial\Fix(f),\]
where $id$ denotes the identity in $F$, are also finite. Note that by Proposition~\ref{prop_rationals} we have $\cup_{n\geq1}D_n=\mathbb Q$.

Let $x$ be an arbitrary irrational point from $[0,1]$. In this case $x$ does not belong to $D_{2n}$ and since $D_{2n}$ is finite there is some neighborhood $U$ of $x$ without points from $D_{2n}$. We will show that for any point $y\in U$
\begin{equation}
d(\Psi(x),\Psi(y))\leq 2^{-n}.
\end{equation}
For this we need to show that the Schreier graphs
$\Gamma(F,\Stf(x),\{x_0,x_1\})$ and $\Gamma(F,\Stf(y),\{x_0,x_1\})$
agree on the balls with radius $n$ centered at $x$ and $y$
respectively. Suppose this is not the case. Then there exist
$f,g\in F$ of length at most $n$, satisfying $f(x)=g(x)$ and $f(y)\neq g(y)$.
Therefore the element $h=fg^{-1}$ of length at most $2n$ satisfies $h(x)=x$ and $h(y)\neq y$. But that is possible only if $\partial\Fix(h)$ has nontrivial intersection with the segment connecting $x$ and $y$, which contradicts our assumption that $U\cap D_{2n}=\emptyset$ because $\partial\Fix(h)\subset D_{2n}$.
\end{proof}

In the description of Schreier graphs of points in $(0,1)$ we will refer to the the one-sided shift on $X^\omega$, denoted by $\sigma$ and defined by
\[\sigma(x_1x_2x_3\ldots)=x_2x_3x_4\ldots\]
for $x_1x_2x_3\ldots\in X^\omega$.

\begin{theorem}
\label{thm_schreier_irrational}
Each irrational number $x$ in $[0,1]$ can be uniquely written as either $0.1^n0w$ or $0.0^m1w$ for some infinite word $w$ over $X$.
The Schreier graph $\Gamma_x$ of the action of $F$ on the orbit of $x$ is depicted in Figure~\ref{fig_correspondence} (particularly for $x=0.0001w$, where $w=1001\ldots$) and has the following structure:
\begin{itemize}
\item the base vertex is labeled by $x$;
\item each vertex labeled by $10\ast$ (of a grey color) is a root of the tree hanging down from this vertex that is canonically isomorphic to the tree hanging down at the vertex $10w$;
\item an infinite path going from the vertex $10w$ upwards turns left or right at the $k$-th level, if the $k$-th letter in $w$ is 0 or 1 respectively.
\item labels of vertices agree with the standard action of $F$ on the Cantor set~\eqref{eqn_cantor_action}.
\end{itemize}
\begin{figure}
\begin{center}
\epsfig{file=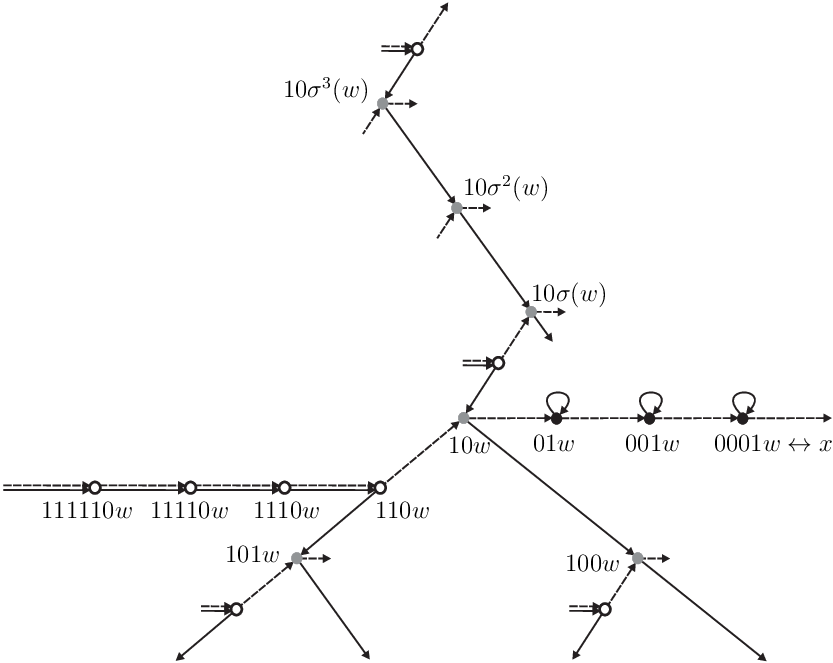,width=350pt}
\caption{Schreier graph $\Gamma_x$ of the action of $F$ on the orbit of an irrational point $x$\label{fig_correspondence}}
\end{center}
\end{figure}
\end{theorem}

\begin{proof}
For an irrational point $x\in X^\omega$ let $m_i$ be the position of the $i$-th 1 in $x$ and let $x_{m_i}$ be the initial segment of $x$ of length $m_i$ (by construction $x_{m_i}$ ends in 1). Then since the sequence $q_i=x_{m_i}0^\infty$ converges to $x$ as $i\to\infty$, by Proposition~\ref{prop:psi_continuous} we have
\[\lim_{i\to\infty}\Gamma_{q_i}=\Gamma_x.\]
Taking longer and longer initial segments of $x$ and constructing Schreier graphs of $q_i$ (considered as a dyadic rational number) corresponds to moving the base vertex in the graph $\Gamma$ in Figure~\ref{fig_schreier_dyadic} deeper and deeper in the tree. The following observation gives a way to understand this more explicitly. It follows from Table~\ref{tab_labeling} that if the label of a grey vertex $v$ in $\Gamma$ is $10w$, then the label of the grey vertex $v'$ right above $v$ (i.e. the closest to $v$ grey vertex located above $v$) is $10\sigma(w)$. Moreover, if the first letter of $w$ is 0, then $v'$ is to the left of $v$, and if it is 1, then $v'$ is to the right of $v$.

Now, suppose that $x=1^n0u1w$ (or $x=0^n1u1w$) for some $u\in X^*$ and $w\in X^\omega$ such that the letter 1 between $u$ and $w$ in $x$ is at position $m_i$. Then $x_{m_i}=1^n0u1$ and the geodesic connecting the base vertex $q_i$ with the root of the tree (vertex $101$) can be described as follows. First one has to move along the line of white vertices to the right (or black vertices to the left) until the vertex $10u1$ of the tree is reached. After this the geodesic goes into the tree containing grey vertices and is represented by the unique path from $10u1$ to $101$. By the observation in the end of previous paragraph, moving along this path from bottom to top corresponds to applying powers of $\sigma$ to $u1$. And this path will turn left or right on the $k$-th step depending on the first letter of $\sigma^{k-1}(u1)$, which is exactly the $k$-th letter in $u1$ and in $x$.

When we increase $i$, we increase the length of the path joining the base vertex $q_i$ with the root of the tree. By construction, when we take a limit of $\Gamma_{q_n}$, we obtain precisely the graph $\Gamma_x$ described in the statement of the theorem. The labels of the vertices of $\Gamma_x$ are defined as the limits of the labels of the vertices in $\Gamma_{q_i}$. In particular, as $\lim_{i\to\infty}q_i=x$, the label of the selected vertex in $\Gamma_x$ is $x$. Since the labels of vertices of $\Gamma_{q_i}$ agree with the standard action~\eqref{eqn_cantor_action} of $F$ on the Cantor set, the labels of vertices in $\Gamma_x$ will also possess this property. Therefore, for each vertex $v$ of $\Gamma_x$ the label of $v$ is just an image of $x$ under the action of elements of $F$ corresponding to any paths joining $x$ and $v$.
\end{proof}

\section{Closure of the space of Schreier graphs of dyadic rational numbers}
\label{sec_closure}

By changing the selected point in the Schreier graph depicted in Figure~\ref{fig_schreier_dyadic} we obtain a family $\mathcal C$ of pointed graphs. Thompson's group $F$ acts on this set by a shift of the base point according to the labels on the edges. This action of $F$ on the set of pointed Schreier graphs is called a \emph{Schreier dynamical system} associated to the orbit of $\frac12$ (see Section~8 in~\cite{grigorch:dynamics11}). By continuity it can be extended to the closure $\overline{\mathcal C}$
of the set $\mathcal C$ in the topology described above. With a slight abuse of notation we will call the last action also by Schreier dynamical system. Below we will study this action. But first we will describe the structure of the set $\overline{\mathcal C}$.

\begin{proposition}
\label{prop_closure}
The closure $\overline{\mathcal C}$ of the family of Schreier graphs of dyadic rational numbers in $[0,1]$ contains three types of graphs:
\begin{itemize}
\item[A.] Schreier graphs of dyadic rational numbers shown in Figure~\ref{fig_schreier_dyadic} with any point as a base point;
\item[B.] graphs in Figure~\ref{fig_correspondence} with an infinite path going up and turning left or right at every level, where the base vertex has a label either $1^n0w$ or $0^n1w$ for $w\in X^\omega$;
\item[C.] the two graphs shown in Figure~\ref{fig_closure} that are quasi-isometric to a line.
\end{itemize}
\begin{figure}
\begin{center}
\epsfig{file=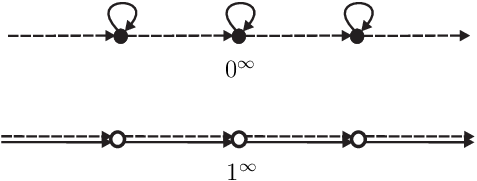,width=180pt}
\caption{Fixed points in the space of pointed Schreier graphs\label{fig_closure}}
\end{center}
\end{figure}
\end{proposition}

Note for the graphs of type C changing the base vertex does not change the isomorphism type of the pointed graph, so each vertex can be considered as a base vertex. Also, it deserves mentioning that the class of graphs of type B includes all Schreier graphs of irrational numbers, but it also includes those graphs in which the infinite path going upwards is eventually periodic. As we will see below these graphs are not Schreier graphs of rational points.

\begin{proof}
First of all, it is easy to see that all of these graphs belong to the closure. Graphs of type A are in $\mathcal C$. The Schreier graphs of irrational points are constructed as limits of graphs in $\mathcal C$. Similarly, the graphs of type B with an eventually periodic infinite ray going up can be approximated by graphs in $\mathcal C$ exactly in the same way. Finally, the graphs of type C are limits of the sequences $\{\Gamma_{0^n}\}_{n\geq1}$ and $\{\Gamma_{1^n}\}_{n\geq1}$ of graphs in $\mathcal C$.

On the other hand, if $\Gamma'$ is a graph in $\overline{\mathcal C}$ that is a limit of a sequence of graphs $\{\Gamma_{q_n}\}_{n\geq1}$ in $\mathcal{C}$, then the graphs $\Gamma_{q_n}$ have to have bigger and bigger balls centered at $q_n$ that are isomorphic. Therefore, taking into account the observation in the proof of Theorem~\ref{thm_schreier_irrational} relating the letters in $q_n$ and the structure of $\Gamma_{q_n}$, the points $q_n$ must have longer and longer common beginnings. This, in turn, forces $\Gamma'$ to be in one of the classes described in the formulation of the theorem.

Similarly to the proof of Theorem~\ref{thm_schreier_irrational} we can define the vertex labels for each graph $\Gamma'$ of type B as the limits of vertex labels of the graphs in $\mathcal C$ that approach this graph. In view of the argument in the previous paragraph, these labels do not depend on the sequence used to approach $\Gamma'$. Again, since the labels of vertices of graphs of type B agree with the standard action~\eqref{eqn_cantor_action} of $F$ on the Cantor set, the labels of vertices in $\Gamma'$ will also have this property. Hence, for each vertex $v$, its label represents the image of the label of the base vertex under the action of an element of $F$ corresponding to any paths connecting the base vertex to $v$.

The graph of type B where the selected vertex has label $w$ will be denoted by $\widetilde\Gamma_w$ (so that for irrational $w$ we have $\widetilde\Gamma_w=\Gamma_w$). We will also denote the graphs of type C by $\widetilde\Gamma_{0^\infty}$ (the top one) and $\widetilde\Gamma_{1^\infty}$ (the bottom one).

Finally, if $w\neq w'$ then $\widetilde\Gamma_w$ is not isomorphic to $\widetilde\Gamma_w'$ as pointed graphs. Indeed, if the first digit where $w$ differs from $w'$ is at position $m$, then these graphs have non-isomorphic $(2m-5)$-neighborhoods of the selected points.
\end{proof}

To emphasize the difference between $\widetilde\Gamma_w$ and $\Gamma_w$ for a rational $w$ we note that in this case there are multiple vertices in $\widetilde\Gamma_w$ with the same label. This fact will allow us in Section~\ref{sec_schreier_rational} to construct the Schreier graph $\Gamma_w$ of a rational point $w$ by gluing all vertices with the same labels in the graph $\widetilde\Gamma_w$.

In order to better understand the action of $F$ on $\overline{\mathcal C}$ we give more details about the structure of $\overline{\mathcal C}$. It turns out that the space $\overline{\mathcal C}$ is not perfect, i.e. it does not coincide with the set of its limit points. The ordinal that shows how different the topological space $X$ is from the perfect space is called the \textit{Cantor-Bendixson rank} of $X$. It is defined as the ``number'' of times one has to throw away isolated points in $X$ to get a perfect set. The perfect set that is obtained in the end of this procedure is called the \textit{perfect kernel} of $X$.

The following proposition answers Question~6.1 in~\cite{grigorch:dynamics11} for the space of pointed Schreier graphs of $F$, and allows us to look at the standard action of $F$ on the Cantor set from different perspective.

\begin{proposition}
\label{prpo_cantor_bendixson}
The Cantor-Bendixson rank of $\overline{\mathcal C}$ is equal to 1. The perfect kernel $\mathcal D$ of the set $\overline{\mathcal C}$ consists of graphs of type B and C from Proposition~\ref{prop_closure}, and is homeomorphic to the Cantor set.
\end{proposition}

\begin{proof}
First of all we show that the graphs of type A are isolated points in $\overline{\mathcal C}$. Indeed, let $\Gamma_q$ be an arbitrary graph of type A. This means that it is a Schreier graph of some dyadic rational point $q$. The set of vertices $V=\{1,11,111\}$ in $\Gamma$ has the property that the subgraph $\Gamma_V$ of $\Gamma$ induced by $V$ (depicted in Figure~\ref{fig_unique_subgraph}) is not isomorphic to any other subgraph of $\Gamma$, or any subgraph of graphs of type B or C, spanned by 3 vertices. So each automorphism of $\Gamma$ fixes the vertices in $V$.  Moreover, the identity map on $\Gamma_V$ uniquely extends to the trivial automorphism of $\Gamma$ as of a marked graph. Thus, the automorphism group of $\Gamma$ as of an abstract marked graph is trivial. Therefore, if we take a ball in $\Gamma_q$ centered at $q$ of a radius $R$ that is sufficiently big to include set $V$, it will not be isomorphic to any other ball of the same radius in any other graph in $\overline{\mathcal C}$. Thus, the $2^{-R}$-neighborhood of $\Gamma_q$ in $\overline{\mathcal C}$ contains only $\Gamma_q$ and, hence, $\Gamma_q$ is an isolated point.

\begin{figure}
\begin{center}
\epsfig{file=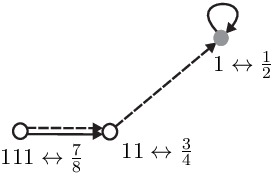, width=100pt}\\
\parbox{5in}{\caption{The induced subgraph $\Gamma_V$ of $\Gamma$ not isomorphic to any other subgraph spanned by 3 vertices\label{fig_unique_subgraph}}}
\end{center}
\end{figure}

On the other hand, let $\widetilde\Gamma_w$ be a graph in $\overline{\mathcal C}$ of type B. Let $w'$ be an irrational string such that $w'\neq \sigma^n(w)$ for all $n\geq0$. For any $n\geq 0$ define
\[v_n=w_nw'\in X^\omega,\]
where $w_n$ denotes the initial segment of $w$ of length $n$. Then for the sequence $\{\widetilde\Gamma_{v_n}\}$ of graphs of type B, by construction of graphs in $\overline{\mathcal C}$ we have
\[\lim_{n\to\infty}\widetilde\Gamma_{v_n}=\widetilde\Gamma_w.\]
Therefore, $\widetilde\Gamma_w$ is not an isolated point in $\overline{\mathcal C}$.

The same argument applied to graphs of type $C$ shows that the set of isolated points in $\overline{\mathcal C}$ precisely coincides with the set of pointed graphs of type A. Moreover, as shown above, the graphs approximating graphs of types B and C can be chosen to be of type B. Thus, the set $\mathcal D$ obtained from $\overline{\mathcal C}$ by removing graphs of type A does not contain isolated points, and, hence, is a perfect kernel of $\overline{\mathcal C}$.

Finally, the homeomorphism between $\mathcal D$ and the Cantor set $X^\omega$ is given by sending the pointed graph to the label of the selected vertex. This map is clearly bijective and continuous by construction of graphs in $\mathcal D$.
\end{proof}

\begin{theorem}
\label{cor_action_conjug}
The action of $F$ on $\mathcal D$ induced by the Schreier dynamical system is topologically conjugate to the standard action of $F$ on the Cantor set $X^\omega$.
\end{theorem}

\begin{proof}
Let $\psi\colon \mathcal D\to X^\omega$ be a homeomorphism defined in the proof of Proposition~\ref{prpo_cantor_bendixson} that sends the pointed graph in $\mathcal D$ to the label of its selected vertex. Let also $\widetilde\Gamma_w\in\mathcal D$ and $f\in F$ be arbitrary. Then, as was proven in Proposition~\ref{prop_closure}, the label of the selected vertex of $\widetilde\Gamma_w^f$ is equal to the image of $w$ under $f$. Therefore,
\[\psi\bigl(\widetilde\Gamma_w\bigr)^f=w^f=\psi\bigl(\widetilde\Gamma_w^f\bigr).\]
Thus, the actions of $F$ on $\mathcal D$ and on $X^\omega$ are topologically conjugate.\end{proof}

\section{Schreier graphs of rational numbers and applications}
\label{sec_schreier_rational}

Theorem~\ref{cor_action_conjug} tells us that we have not construct any new action of $F$ on the Cantor set. However it gives us more information about this action. Namely, it allows us to describe the Schreier graphs of the action of $F$ on the orbit of every point of $X^{\omega}$.

\begin{figure}[h]
\begin{center}
\epsfig{file=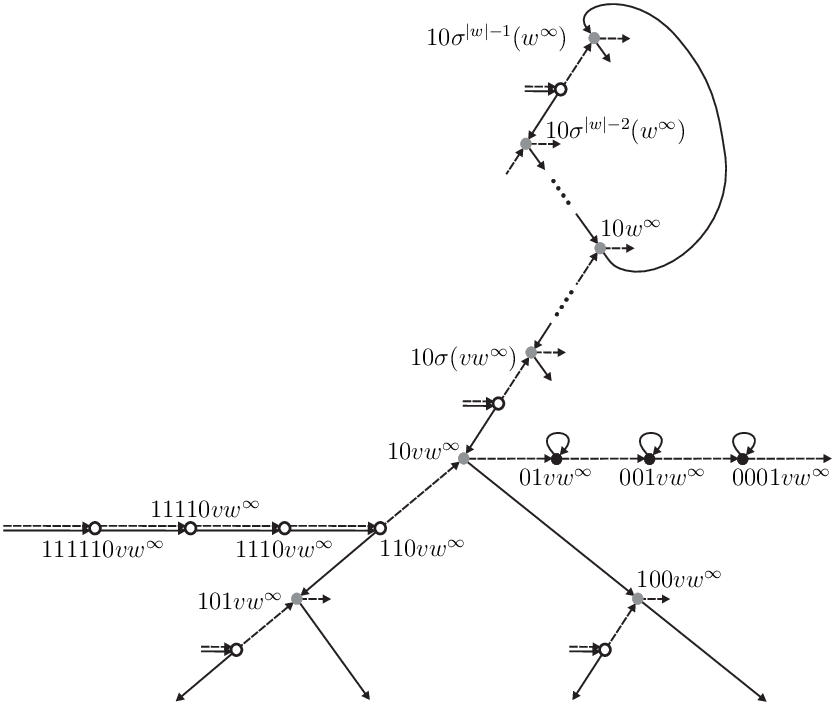,
width=350pt}
\parbox{5in}{\caption{Schreier graph of the action of $F$ on the orbit of a rational point of a Cantor set\label{fig_correspondence_rational}}}
\end{center}
\end{figure}

\begin{theorem}
\label{thm_schreier_rational}
Each rational point $x$ of the Cantor set $X^\omega$ except $000\ldots$ and $111\ldots$ can be uniquely written as either $x=1^n0vwww\ldots$ or $x=0^m1vwww\ldots$ for some $v,w\in X^*$ such that $w$ is not a proper power and the ending of $v$ differs from the one of $w$. The Schreier graph $\Gamma_x$ of the action of $F$ on the orbit of $x$ is depicted in Figure~\ref{fig_correspondence_rational} and has the following structure:
\begin{enumerate}
\item the base point is labeled by $x$;
\item each vertex labeled by $10u$ for $u\neq \sigma^{i}(w^\infty)$ is a root of the tree hanging down from this vertex that is canonically isomorphic to the tree hanging down at the vertex $101$ in the Schreier graph $\Gamma$ of $\frac12$ shown in Figure~\ref{fig_schreier_dyadic};
\item a path connecting vertices $10vw^\infty$ and $10\sigma^{|w|-1}(w^\infty)$ turns left or right at the $k$-th level, if the $k$-th letter in $vw$ is 0 or 1 respectively.
\item if the first letter of $w$ is 0 then there is an edge from the vertex $10w^\infty$ to the vertex $10\sigma^{|w|-1}(w^\infty)$ labeled by $x_1$ (see Figure~\ref{fig_correspondence_rational});
\item if the first letter of $w$ is 1 then the vertex $110w^\infty$ is adjacent to vertices $10w^\infty$ and $10\sigma^{|w|-1}(w^\infty)$ via edges labeled by $x_0$ and $x_1$ respectively (see the example in Figure~\ref{fig_schreier13});
\item the path in the third item and the edge or the path in the previous two items create a loop in the Schreier graph corresponding to $w$ which we will call a \emph{nontrivial loop};
\item each vertex adjacent to the grey vertex in the nontrivial loop, and not belonging to this loop is either
\begin{itemize}
\item the beginnings of geodesics isomorphic to geodesic $(1,01,001,\ldots)$ in $\Gamma$, or
\item the root of the tree hanging down and to the left at the vertex $11$ in $\Gamma$, or
\item the root of the tree hanging down and to the right at the vertex $101$ in $\Gamma$.
\end{itemize}
\end{enumerate}
\end{theorem}

\begin{proof}
By Theorem~\ref{cor_action_conjug} the Schreier graph $\Gamma_x$ of a rational point $x\in X^\omega$ coincides with the Schreier graph of the action of $F$ on the orbit of $\widetilde\Gamma_x$ in $\mathcal{D}$. Therefore, it is obtained from $\widetilde\Gamma_x$  by gluing the vertices with identical labels. Since $x$ is rational, an infinite path in $\widetilde\Gamma_x$ going upwards will be eventually periodic with the period corresponding to the period $w$ of $x$. The labels of grey vertices along this path will repeat with period $|w|$ as soon as iterations of the shift $\sigma$ erase $v$ along the path connecting vertices $10vw^\infty$ and $10w^\infty$. Gluing corresponding vertices along this path, together with all subtrees hanging down from these vertices, creates the graph precisely described in the statement. In particular, the grey vertices with the labels $10w^\infty$ through $10\sigma^{|w|-1}(w^\infty)$ create a nontrivial loop in the Schreier graph.

Note that since the words $v$ and $w$ have different endings, the edge in item 4 and the path in item 5 are well defined and do not lead to the situation where there are two edges labeled by the same generator having the same initial vertex, or the same terminal vertex.
\end{proof}

\begin{figure}[h]
\begin{center}
\epsfig{file=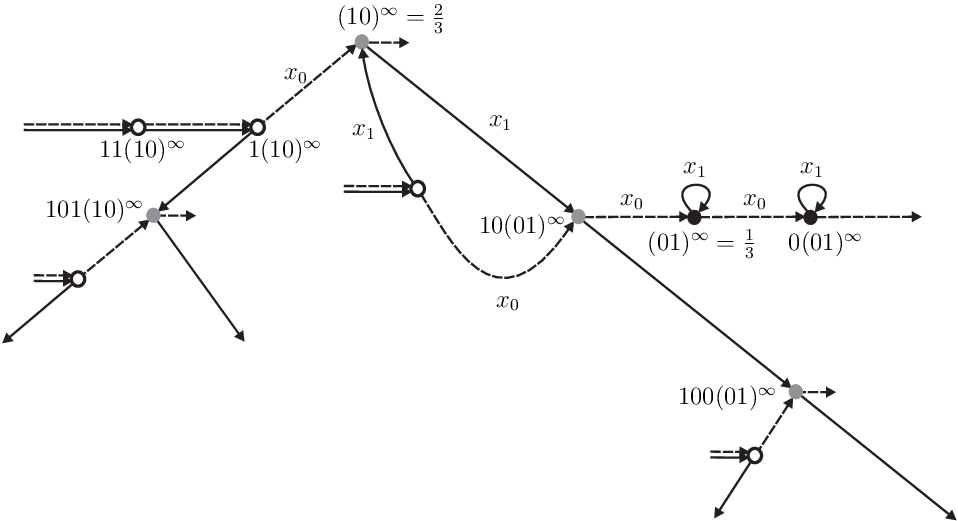, width=300pt}
\caption{Schreier graph of the action of $F$ on the orbit of $\frac13$\label{fig_schreier13}}
\end{center}
\end{figure}

For example, Schreier graph of the action of $F$ on the orbit of $\frac13$ is shown in Figure~\ref{fig_schreier13}.

\begin{remark}
For a rational point $q\in (0,1)$ the Schreier graph $\Gamma_q$ gives a way to find the shortest element stabilizing $q$ and not stabilizing its neighborhood. Every representation of such element as a word in generators of $F$ creates a loop in $\Gamma_q$ starting and ending at $q$ and going through the ``nontrivial loop'' in the graph. For example, for $q=\frac13=0.(01)^\infty$ the shortest such element is $x_0^{-1}x_1^{-2}x_0^2$. The graph of this element is depicted in Figure~\ref{fig_graph13}.
\end{remark}

\begin{remark}
If we apply the construction from Theorem~\ref{sec_schreier_rational} to dyadic rational numbers, we will recover exactly the graph $\Gamma$ from Figure~\ref{fig_schreier_dyadic}. Indeed, the number of grey vertices in the nontrivial loop corresponding to the period of a binary expansion of a rational element $x$ is equal to the length of the period. In the case of a dyadic rational number, the length of the period is 1, so there is just one grey vertex (labeled by $10^\infty=\frac12$) in the nontrivial loop, which makes this loop degenerate. In other words, we could have started from the set $\mathcal D$ on which $F$ acts according to the standard action, and from there get all Schreier graphs. This observation suggests a possible generalization of the method described here to construct Schreier graphs of the actions of other groups on the Cantor set.

\end{remark}

As direct corollaries of Theorems~\ref{thm_schreier_irrational} and~\ref{thm_schreier_rational} we have the following statements. Recall that for a graph $X$ the number of ends of $X$ is defined as the supremum of the number of unbounded components of $X-C$, where $C$ runs over all finite subgraphs of $X$.

\begin{corollary}
The Schreier graph of the action of $F$ on the orbit of any point in $(0,1)$ has infinitely many ends.
\end{corollary}

\begin{figure}[h]
\begin{center}
\epsfig{file=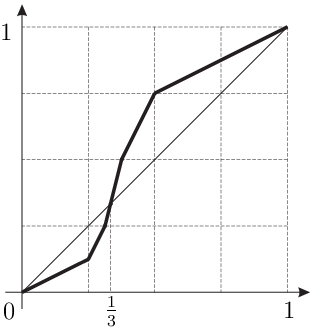, width=130pt}\\
\parbox{5in}{\caption{Graph of the shortest element stabilizing $\frac13$ and not stabilizing any neighborhood of $\frac13$\label{fig_graph13}}}
\end{center}
\end{figure}

\begin{corollary}
\label{cor_noniso}
Different points in $(0,1)$ have non-isomorphic pointed marked Schreier graphs.
\end{corollary}

\begin{proof}
Suppose $\Gamma_x$ and $\Gamma_y$ are isomorphic and let $\phi\colon\Gamma_x\to\Gamma_y$ be an isomorphism. First we assume that $x$ and $y$ are irrational. We can further assume that $x$ is a grey vertex in $\Gamma_x$. Let $v_0=x$. Since 2-neighborhood $B(v_0,2)$ of $v_0=x$ must be isomorphic to 2-neighborhood $B(\phi(v_0),2)$ of $\phi(v_0)=y$, $\phi(v_0)$ also has to be a grey vertex in $\Gamma_y$. Moreover, there is only one way to map $B(v_0,2)$ to $B(\phi(v_0),2)$. Let $v_1$ be the grey vertex in $B(v_0,2)$ right above $v_0$. Then $\phi(v_1)$ is the grey vertex right above $\phi(v_0)$. Moreover, if $v_1$ was to the right (left) of $v_0$, $\phi(v_1)$ will be to the right (left) of $\phi(v_0)$. Similarly, we denote by $v_{n+1}$ the grey vertex in $\Gamma_x$ right above $v_n$ for $n\geq 1$ and show that the image of the path $P=(x=v_0, v_1,\ldots,v_{n+1})$ under $\phi$ goes from $\phi(v_0)=y$ up to $\phi(v_{n+1})$ and is parallel to $P$. Thus, by Theorem~\ref{thm_schreier_irrational} we must have $x=y$.

The case when both $x$ and $y$ are rational is considered similarly. Finally, the graphs $\Gamma_x$ and $\Gamma_y$ for rational $x$ and irrational $y$ cannot be isomorphic.
\end{proof}

If we consider non-pointed Schreier graphs (without selected vertex), then the isomorphism class of $\Gamma_x$ will depend only on the orbit of $x$. More precisely,

\begin{corollary}
The following conditions are equivalent:
\begin{enumerate}
\item[(a)] the marked Schreier graph of points $x$ and $y$ in $(0,1)$ are isomorphic as abstract marked graphs (without selected vertex);
\item[(b)] $x$ and $y$ belong to the same orbit of $F$;
\item[(c)] $x$ and $y$ have binary expansions with the same tails.
\end{enumerate}
\end{corollary}

Note, that the equivalence of (b) and (c) was proved also by Belk and Matucci (see~Proposition 3.2.3 in~\cite{matucci:phd08} or Proposition 2.4 in the second revision of the preprint~\cite{belk_m:dynamics_of_F}).

\begin{proof}
The fact that (b) implies (c) follows from the definition of the standard action of $F$ on the Cantor set~\eqref{eqn_cantor_action}. The converse of this implication follows from the structure of Schreier graphs described in Theorems~\ref{thm_schreier_irrational} and~\ref{thm_schreier_rational}. Namely, given a binary expansion of $x$, one can explicitly construct a vertex in $\Gamma_x$ whose label is a binary expansion of $y$. Indeed, first one can move from $x$ to some of the grey vertices, and then moving up the tree (and possibly around the nontrivial loop in the graph) corresponds to applying shift $\sigma$ to the label of a vertex in the path. Since expansions of $x$ and $y$ have the same endings, at some point we will reach the vertex whose label labeled by binary expansion, which is an ending of a binary expansion of $y$. From this point one can move down the tree according to the beginning digits of $y$ to the vertex with label $y$.

Further, (a) follows directly from (b). The converse ((a) implies (c)) easily follows from the proof of Corollary~\ref{cor_noniso}.
\end{proof}

Another corollary of Theorems~\ref{thm_schreier_irrational} and~\ref{thm_schreier_rational} is that the Schreier dynamical system built from the action of $F$ on the orbit of $x\in(0,1)$ essentially does not depend on $x$. More precisely, the perfect kernel of the closure of the space of pointed Schreier graphs $\Gamma_y$, where $y$ is in the orbit of $x$, coincides with the set $\mathcal D$. Therefore, we obtain

\begin{corollary}
The action of $F$ on the perfect kernel $\mathcal D$ of the Schreier dynamical system associated to the orbit of $x\in(0,1)$ is topologically conjugate to the standard action of $F$ on the Cantor set $X^\omega$.
\end{corollary}

\begin{proof}
Completely analogous to the proof of Theorem~\ref{cor_action_conjug}.
\end{proof}

Finally, we note that our approach does not answer the question of amenability. Namely, the following proposition holds.

\begin{proposition}
Let $\{d_1,\ldots,d_n\}$ be a finite subset of the Cantor set. Then
the Schreier graph $\Gamma(F,\Stf(d_1,\ldots,d_n),\{x_0,x_1\})$ is
amenable.
\end{proposition}

\begin{proof}
The proof is identical to the proof of the analogous proposition in~\cite{savchuk:thompson} (Proposition 3). The amenability of these graphs is given by the fact that all of them possess arbitrarily long line segments with a boundary of size 2.
\end{proof}


\def\cprime{$'$}

\noindent email: \href{mailto:savchuk@usf.edu}{savchuk@usf.edu}\\
Department of Mathematics and Statistics\\University of South Florida\\
               4202 E Fowler Ave\\
               Tampa, FL 33620-5700

\end{document}